\documentclass[12pt]{article}%
\usepackage{amssymb}
\usepackage{amsfonts}
\usepackage{amsmath}
\usepackage{color}
\usepackage{graphicx}%
\providecommand{\U}[1]{\protect\rule{.1in}{.1in}}
\newtheorem{theorem}{Theorem}

\newtheorem{definition}[theorem]{Definition}

\newtheorem{lemma}[theorem]{Lemma}

\newtheorem{remark}[theorem]{Remark}

\newenvironment{proof}[1][Proof]{\noindent\textbf{#1.} }{\ \rule{0.5em}{0.5em}}
\numberwithin{equation}{section}
\begin{document}

\title{A variational principle for nonpotential perturbations of gradient flows of
nonconvex energies \thanks{\textbf{Acknowledgment.}\quad\textrm{This work has
been supported by the Austrian Science Fund (FWF) project P27052-N25.}}}
\author{{Stefano Melchionna\thanks{ Faculty of Mathematics, University of Vienna,
Oskar-Morgenstern-Platz 1, 1090 Wien, Austria. $\qquad$ E-mail:
\textit{stefano.melchionna@univie.ac.at}} }}
\maketitle

\begin{abstract}
We investigate a variational approach to nonpotential perturbations of
gradient flows of nonconvex energies in Hilbert spaces. We prove existence of
solutions to elliptic-in-time regularizations of gradient flows by combining
the minimization of a parameter-dependent functional over entire trajectories
and a fixed-point argument. These regularized solutions converge up to
subsequence to solutions of the gradient flow as the regularization parameter
goes to zero. Applications of the abstract theory to nonlinear
reaction-diffusion systems are presented.

\end{abstract}

\section{Introduction}

This work is concerned with a \textit{nonpotential perturbation} of a gradient
flow driven by a possibly nonconvex energy $\phi:H\rightarrow(-\infty
,+\infty]$, namely%
\begin{align}
u^{\prime}+\mathrm{D}\phi(u) &  \ni f(u)\qquad\text{a.e. in }%
(0,T),\label{grdaient flow}\\
u(0) &  =u_{0}\text{.}\label{grdaient flow ic}%
\end{align}
Here $H$ is a real Hilbert space, $u^{\prime}$ denotes the time derivative of
$u$, and we assume that the energy $\phi$ can be decomposed as
\[
\phi=\varphi_{1}-\varphi_{2},\quad\varphi_{1},\varphi_{2}:H\rightarrow
(-\infty,+\infty]\text{,}%
\]
where $\varphi_{1},\varphi_{2}$ are proper, bounded from below, and lower
semicontinuous (l.s.c.) functionals. The symbol $\mathrm{D}\phi$ represents
some suitably-defined gradient of the functional $\phi$ (see below),
$f:H\rightarrow H$ is an (at most) linearly growing and continuous function,
and $u_{0}\in D(\phi):=\{u\in H:\phi(u)<+\infty\}$. Note that we are not
assuming here $f(u)=\mathrm{D}F(u)$ for some $F:H\rightarrow%
%TCIMACRO{\U{211d} }%
%BeginExpansion
\mathbb{R}
%EndExpansion
$. In particular, the perturbation term $f$ is \textit{nonpotential.
}Including a nonpotential term allows us to apply our theory to systems of
differential equations, see Section \ref{applications}.

Gradient flows arise ubiquitously in connection with dissipative evolution and
correspond to problem (\ref{grdaient flow}) for $f=0$, namely,
\begin{align}
u^{\prime}+\mathrm{D}\phi(u)  &  \ni0 \qquad\text{a.e. in }%
(0,T),\label{np grad flow}\\
u(0)  &  =u_{0}\text{.}%
\end{align}
Together with its nonpotential perturbation (\ref{grdaient flow}), the latter
describes a variety of dissipative evolution situations and it is therefore
crucially relevant in applications.

A recent variational approach to dissipative problems is the so-called
\textit{weighted energy-dissipation }(WED) procedure. This consists in
defining an energy-dissipation functional $I_{\varepsilon}$ over entire
trajectories which depends on a parameter $\varepsilon$ and prove that its
minimizers converge to solutions to the target problem for $\varepsilon
\rightarrow0$. Such a global-in-time variational approach to dissipative
problems is interesting, since it paves the way to the application of tools
and techniques of the calculus of variation (e.g. Direct Method, relaxation,
$\Gamma$-convergence). Moreover, the WED procedure brings also a new tool to
check qualitative properties of solutions and comparison principles for
dissipative problems. A detailed discussion of this application will appear in
a forthcoming paper. In addition, the minimization problem features,
typically, more regular solutions. This is indeed the case here, as the
Euler-Lagrange system associated with the minimization of the WED functional
corresponds to an elliptic-in-time regularization of the gradient flow
problem. The elliptic-regularization approach to evolution equations has to be
traced back at least to \cite{Li} and \cite{Ol}, see also \cite{Li-Ma}. A
first occurrence of the WED functional approach is in \cite{Il} and \cite{Hi}.
Later, the WED formalism has been reconsidered by Mielke and Ortiz
\cite{Mi-Or} for rate-independent equations. The gradient flow case with
$\lambda$-convex potentials has been studied by Mielke and Stefanelli
\cite{Mi-St}. The extension to the genuinely nonconvex energy case is due to
Akagi and Stefanelli \cite{Ak-St}. Finally, \cite{Ak-St2} and \cite{Ak-St3}
are concerned with the WED functional for doubly nonlinear problems.

This note extends the WED variational approach to the nonpotential case. In
particular, the results from \cite{Mi-St} and \cite{Ak-St} will be recovered.
In addition, our new technique will allow the application of the method to
systems of gradient flows, e.g., of reaction diffusion equations coupled via
the reaction terms. We remark that existence results for the Cauchy problem
(\ref{grdaient flow})-(\ref{grdaient flow ic}) have already been proved in
\cite{At-Da,Ot}. The main result of this work is that solutions to
(\ref{grdaient flow})-(\ref{grdaient flow ic}) can be obtained as limits of
solutions to an elliptic regularization of (\ref{grdaient flow}%
)-(\ref{grdaient flow ic}), which is tackled by combining a fixed-point
argument and a variational technique. Having already observed that
$f(u)\neq\mathrm{D}F(u)$ for all $F:H\rightarrow%
%TCIMACRO{\U{211d} }%
%BeginExpansion
\mathbb{R}
%EndExpansion
$, we intend to use here a variational technique in order to solve a problem
which has no variational nature. We do this by combining the WED approach with
a fixed-point argument. More precisely, we define an operator $S:L^{2}%
(0,T;H)\rightarrow L^{2}(0,T;H)$ through the minimization of the WED-type
functional by letting%
\begin{align}
S &  :v\mapsto u=\arg\min_{w}I_{\varepsilon,v}(w),\nonumber\\
I_{\varepsilon,v}(w) &  =\int_{0}^{T}\exp(-t/\varepsilon)\left(
\frac{\varepsilon}{2}|w^{\prime}|^{2}+\phi(w)-(f(v),w)\right)  \mathrm{d}%
t,\nonumber
\end{align}
and we check that $S$ has a fixed point which satisfies an elliptic-in-time
regularization of equation (\ref{np grad flow}) (cf. Theorem \ref{main teo}
and Theorem \ref{full gen th}):%
\begin{align}
-\varepsilon u_{\varepsilon}^{\prime\prime}+u_{\varepsilon}^{\prime}%
+\partial\phi(u_{\varepsilon}) &  \ni f(u_{\varepsilon})\text{ \ \ \ \ \ a.e.
in }(0,T),\label{eu la}\\
u_{\varepsilon}(0) &  =u_{0}\text{, \ \ \ }u_{\varepsilon}^{\prime
}(T)=0\text{.}\label{eu la 2}%
\end{align}
Then, by passing to the limit $\varepsilon\rightarrow0$ we recover a solution
to equation (\ref{grdaient flow}).

In Section \ref{assumptions} we enlist the assumptions which are assumed
throughout the paper and we state our main results. We first prove our results
in the simpler case of a convex potential $\phi$ in Section
\ref{lambda convex energy}. In Section \ref{nonconvex energy} we prove the
results in full generality, namely we deal with the case of nonconvex energies
$\phi$. Section \ref{generalization} illustrates how to generalize our results
to the case of less regular initial data. Finally, we present applications of
our abstract theory to reaction-diffusion systems in Section
\ref{applications}.

\section{Assumptions and main results\label{assumptions}}

We enlist here the assumptions which are considered throughout the paper. Let
$H$ be a real Hilbert space with scalar product $(\cdot,\cdot)$ and norm
$|\cdot|$. Let the function $f:H\rightarrow H$ be continuous and sublinear,
namely
\begin{equation}
|f(u)|\leq C_{1}(1+|u|)\label{linear growth f}%
\end{equation}
for all $u\in H$ and some positive constant $C_{1}$. We assume that the
functional $\phi$ can be decomposed as $\phi=\varphi_{1}-\varphi_{2}$, where
$\varphi_{1},\varphi_{2}:H\rightarrow\lbrack0,+\infty]$ are bounded from
below, proper, l.s.c., and convex functionals. Furthermore, we assume
$D(\varphi_{1})\subset D(\varphi_{2})$, $D(\partial\varphi_{1})\subset
D(\partial\varphi_{2})$, and that there exist constants $k_{1},k_{2}\in
\lbrack0,1),$ $C_{2}>0$, and a non-decreasing function $\ell:%
%TCIMACRO{\U{211d} }%
%BeginExpansion
\mathbb{R}
%EndExpansion
\rightarrow\lbrack0,+\infty)$ such that
\begin{equation}
\varphi_{2}(u)\leq k_{1}\varphi_{1}(u)+C_{2}\label{controll 1}%
\end{equation}
for all $u\in D(\varphi_{1})$ and
\begin{equation}
|\xi|^{2}\leq k_{2}|\left(  \partial\varphi_{1}(u)\right)  ^{\circ}|^{2}%
+\ell(|u|)(\varphi_{1}(u){+}1)\label{controll 2}%
\end{equation}
for all $u\in D(\partial\varphi_{2})$ and $\xi\in\partial\varphi_{2}(u)$. Here
$\partial\varphi_{1},\partial\varphi_{2}$ denote the subdifferentials of
$\varphi_{1}$ and $\varphi_{2}$ respectively and $(\partial\varphi
_{1}(u))^{\circ}$ the element of $\partial\varphi_{1}(u)$ with minimal norm.
Moreover, let $\left(  X,|\cdot|_{X}\right)  $ be a Banach space compactly
embedded in $H$ such that
\begin{equation}
\varphi_{1}(u)\geq c_{X}|u|_{X}^{2}-C_{3}\label{compact assumption}%
\end{equation}
for all $u\in D(\phi)$ and some strictly positive constants $c_{X}$ and
$C_{3}$.

\begin{remark}
We remark that these assumptions are standard and general enough to include a
variety of different problems (cf., e.g., \emph{\cite{Ak-St, Ot, Ot2}} and Section
\emph{\ref{applications}}).
\end{remark}

We are interested in problem%
\begin{align}
u^{\prime}+\partial\varphi_{1}\left(  u\right)  -\partial\varphi_{2}\left(
u\right)   &  \ni f(u) \qquad\text{a.e. in }(0,T),\label{nonconv gf-1}\\
u(0)  &  =u_{0}. \label{nonconv gf}%
\end{align}
Strong solutions to problem (\ref{nonconv gf-1})-(\ref{nonconv gf}) are
defined as follows.

\begin{definition}
[Strong solution]\label{defi} Let the above assumptions be satisfied and
$u_{0}\in D(\varphi_{1})$. Then, $u\in H^{1}(0,T;H)$ is a \emph{strong
solution} to \emph{(\ref{nonconv gf-1})-(\ref{nonconv gf})} if $u(t)\in
D\left(  \partial\varphi_{1}\right)  $ for a.e. $t\in(0,T)$ and it satisfies
\begin{align}
u^{\prime}+\xi &  =f(u)+\eta\text{ \ \ \ a.e. in }%
(0,T),\label{non convex perturbed gf}\\
\xi &  \in\partial\varphi_{1}(u)\text{ \ \ \ a.e. in }(0,T),\\
\eta &  \in\partial\varphi_{2}(u)\text{ \ \ \ a.e. in }(0,T),\\
u(0)  &  =u_{0}\text{,}%
\end{align}
for given $\xi,\eta\in L^{2}(0,T;H)$.
\end{definition}

The main result of this work is the following theorem whose proof is detailed
in Section \ref{nonconvex energy}.

\begin{theorem}
[Elliptic regularization]\label{full gen th}Let $u_{0}\in D(\partial
\varphi_{1})$. Then, the regularized problem%
\begin{align}
-\varepsilon u_{\varepsilon}^{\prime\prime}+u_{\varepsilon}^{\prime}%
+\xi_{\varepsilon} &  =f(u_{\varepsilon})+\eta_{\varepsilon}\text{ \ \ \ a.e.
in }(0,T),\label{euler eq 2}\\
\xi_{\varepsilon} &  \in\partial\varphi_{1}(u_{\varepsilon})\text{ \ \ \ a.e.
in }(0,T),\\
\eta_{\varepsilon} &  \in\partial\varphi_{2}(u_{\varepsilon})\text{ \ \ \ a.e.
in }(0,T),\label{euler eq 2.2}\\
u_{\varepsilon}(0) &  =u_{0}\text{, \ \ \ \ \ }u_{\varepsilon}^{\prime
}(T)=0\label{euler eq 2 ic}%
\end{align}
admits (at least) a solution $u_{\varepsilon}\in H^{2}(0,T;H)$ for
$\varepsilon>0$ small enough. Furthermore, there exist a sequence
$\varepsilon_{n}\rightarrow0$ such that $u_{\varepsilon_{n}}\rightarrow u$
weakly in $H^{1}(0,T;H)$ and strongly in $C\left(  [0,T];H\right)  $ and $u$
is a strong solution of \emph{(\ref{nonconv gf-1})-(\ref{nonconv gf})}.
\end{theorem}

Theorem \ref{full gen th} extends the former analysis from \cite{Ak-St} and
from \cite{Mi-St}, as our theory applies to the nonpotential perturbations.
And it is worth mentioning that solutions to both problem (\ref{nonconv gf-1}%
)-(\ref{nonconv gf}) and problem (\ref{euler eq 2})-(\ref{euler eq 2 ic})
might be nonunique. Even in the case $f=0$, we provide an alternative proof of
the results in \cite{Ak-St}. Note however that assumption
(\ref{compact assumption}) is not required in \cite{Mi-St} and it is replaced
by a weaker one in \cite{Ak-St}, namely the exponent $2$ in
(\ref{compact assumption}) is replaced by $p\geq1$. On the other hand,
(\ref{compact assumption}) is necessary in order to apply the Gronwall Lemma
\ref{gronwall} which is one of the main technical tool of this work.
Additionally, we can prove similar results also in the case $u_{0}\in
D(\varphi_{1})$ (cf. Section \ref{generalization}). More precisely, we
approximate $u_{0}\in D(\varphi_{1})$ by a sequence $u_{0\varepsilon}\in
D(\partial\varphi_{1})$, we solve equation (\ref{euler eq 2}%
)-(\ref{euler eq 2.2}) coupled with $u_{\varepsilon}(0)=u_{0\varepsilon}$ and
$u_{\varepsilon}^{\prime}(T)=0$ for all $\varepsilon$ small enough, and we
pass to the limit $\varepsilon\rightarrow0$.

\section{Convex energy\label{lambda convex energy}}

Before moving to the proof of the main result in full generality, let us
present the argument in the simpler case of convex energy. In particular,
throughout this section we assume $\phi$ to be convex, namely $\varphi_{2}=0$
(i.e., $\phi=\varphi_{1}$). Problem (\ref{nonconv gf-1})-(\ref{nonconv gf})
then reads
\begin{align}
u^{\prime}+\xi &  =f(u)\text{ \ \ \ \ a.e. in }(0,T),\label{gradient flow2}\\
\xi &  \in\partial\phi(u)\text{ \ \ \ \ a.e. in }(0,T),\\
u(0)  &  =u_{0}. \label{gradient flow2.2}%
\end{align}
As we mentioned in the introduction $f(u) \neq\partial F(u)$ for any
$F:H\rightarrow%
%TCIMACRO{\U{211d} }%
%BeginExpansion
\mathbb{R}
%EndExpansion
$. As a consequence, system (\ref{gradient flow2})-(\ref{gradient flow2.2}),
as well as its elliptic-in-time regularization, cannot be seen as the
Euler-Lagrange system corresponding to a minimization problem. The strategy to
overcome this obstruction is to combine the WED approach with a fixed-point procedure.

Let us consider the map $S:L^{2}(0,T;H)\rightarrow L^{2}(0,T;H)$, given by
$S:v\mapsto u$ where $u$ is the global minimizer of the functional
$I_{\varepsilon,v}$ defined by%
\begin{equation}
I_{\varepsilon,v}(u)=\int_{0}^{T}\exp(-t/\varepsilon)\left(  \frac
{\varepsilon}{2}|u^{\prime}|^{2}+\phi(u)-(f(v),u)\right)  \mathrm{{d}%
}t\label{wed funct}%
\end{equation}
over the convex set $K(u_{0}):=\{u\in H^{1}(0,T;H):u(0)=u_{0}\}$. The main
result of this section is the following.

\begin{theorem}
[Convex case]\label{main teo}Let assumption of Theorem \emph{\ref{full gen th}%
} be satisfied with $\phi=\varphi_{1}$. Then, for all $\varepsilon$ small
enough, the map $S$ has at least one fixed point $u_{\varepsilon
}=S(u_{\varepsilon})$. This satisfies the regularized system%
\begin{align}
-\varepsilon u_{\varepsilon}^{\prime\prime}+u_{\varepsilon}^{\prime}%
+\xi_{\varepsilon}-f(u_{\varepsilon})  &  =0\text{ \ \ \ a.e. in
}(0,T),\label{true euler}\\
\xi_{\varepsilon}  &  \in\partial\phi(u_{\varepsilon})\text{ \ \ \ a.e. in
}(0,T),\\
u_{\varepsilon}^{\prime}(T)  &  =0,\\
u_{\varepsilon}(0)  &  =u_{0}, \label{true euler ic}%
\end{align}
along with the solution $\xi_{\varepsilon} \in L^{2}(0,T;H)$. Moreover, the
solution(s) to the regularized system \emph{(\ref{true euler}%
)-(\ref{true euler ic})} converge(s) (up to subsequences) to (one of) the
solution(s) to the gradient flow problem \emph{(\ref{gradient flow2}%
)-(\ref{gradient flow2.2})} weakly in $H^{1}(0,T;H)$ and strongly in
$C([0,T];H)$ for $\varepsilon\rightarrow0$.
\end{theorem}

\subsection{Preliminary results}

In order to prove Theorem \ref{main teo}, we collect some preliminary results.

For all $v\in L^{2}(0,T;H)$ and for $\varepsilon$ small enough it is proved in
\cite{Mi-St} that there exists a unique minimizer $u \in K(u_{0})$ for the
functional $I_{\varepsilon,v}$ defined by (\ref{wed funct}). In particular,
existence is trivial for every $\varepsilon$, while the uniqueness follows
from uniform convexity for $\varepsilon$ small enough, independently of $v$.
Moreover, $u$ is one of the possibly many solutions to the regularized
problem:
\begin{align}
-\varepsilon u^{\prime\prime}+u^{\prime}+\xi-f(v)  &  =0\text{ \ \ \ \ a.e. in
}(0,T),\label{Euler equation}\\
\xi &  \in\partial\phi(u)\text{ \ \ \ \ \ a.e. in }(0,T),\\
\varepsilon u^{\prime}(T)  &  =0,\\
u(0)  &  =u_{0}.
\end{align}

Using the \textit{maximal regularity estimate}, derived in \cite[Lemma
4.1]{Mi-St}, we have
\begin{equation}
\varepsilon^{2}\left\Vert u^{\prime\prime}\right\Vert _{L^{2}(0,T;H)}%
^{2}+\left\Vert u^{\prime}\right\Vert _{L^{2}(0,T;H)}^{2}+\left\Vert
\xi\right\Vert _{L^{2}(0,T;H)}^{2}+\phi(u(T))\leq C+\left\Vert f(v)\right\Vert
_{L^{2}(0,T;H)}^{2} \label{maxRegEst}%
\end{equation}
and hence%
\begin{equation}
\left\Vert u\right\Vert _{L^{2}(0,T;H)}^{2}\leq C+C\left\Vert f(v)\right\Vert
_{L^{2}(0,T;H)}^{2}, \label{bd}%
\end{equation}
where $C$ denotes a positive constant depending on $|\left(  \partial
\varphi_{1}(u_{0})\right)  ^{\circ}|$. This ensures that the map $S$ is well-defined.

\subsection{Proof of Theorem \ref{main teo}}

The proof of the first part of the theorem follows from an application of the
Schaefer fixed-point Theorem \ref{schaefer} in the Appendix. More precisely,
we check that the map $S$ satisfies the assumptions of Theorem \ref{schaefer}
and hence we prove existence of a fixed point for $S$. In what follows the
symbol $C$ denotes a positive constant possibly depending on $T$, $u_{0}$,
$\phi$, but not on $\varepsilon$ which may vary even within the same line.

\paragraph{The map $S:L^{2}(0,T;H)\rightarrow L^{2}(0,T;H)$ is continuous.}

Let $v_{1},v_{2}\in L^{2}(0,T;H)$ be given and denote by $u_{1}$ and $u_{2}$
the unique minimizers of $I_{\varepsilon,v_{1}}$ and $I_{\varepsilon,v_{2}}$
respectively. Then, by computing the difference between the two corresponding
regularized equations, choosing $w=u_{1}-u_{2}$ as test function and
integrating over $[0,t]$ for $t\in(0,T]$, we get%
\begin{align*}
&  -\varepsilon(w^{\prime}(t),w(t))+\varepsilon\int_{0}^{t}|w^{\prime}%
|^{2}+\int_{0}^{t}(\xi_{1}-\xi_{2},w)+\frac{1}{2}|w(t)|^{2}\\
&  \leq\int_{0}^{t}(f(v_{1})-f(v_{2}),w).
\end{align*}
As $\phi$ is convex, the term $\int_{0}^{t}(\xi_{1}-\xi_{2},w)$ is
nonnegative. Hence,
\begin{align}
-\varepsilon(w^{\prime}(t),w(t))+\varepsilon\int_{0}^{t}|w^{\prime}|^{2}%
+\frac{1}{2}|w(t)|^{2} &  \leq\int_{0}^{t}(f(v_{1})-f(v_{2}),w)\nonumber\\
&  \leq\frac{1}{2}\int_{0}^{t}|f(v_{1})-f(v_{2})|^{2}+\frac{1}{2}\int_{0}%
^{t}|w|^{2}.\label{continuity}%
\end{align}
By applying the Gronwall Lemma \ref{gronwall} from Appendix, we have
\begin{align}
\frac{1}{2}|w(t)|^{2} &  \leq\varepsilon(w^{\prime}(t),w(t))+C\int_{0}%
^{t}|f(v_{1})-f(v_{2})|^{2}\nonumber\\
&  ~~~+C\int_{0}^{t}\varepsilon(w^{\prime},w)+Ct\int_{0}^{t}|f(v_{1}%
)-f(v_{2})|^{2}\nonumber\\
&  \leq C\int_{0}^{t}|f(v_{1})-f(v_{2})|^{2}+\varepsilon(w^{\prime
}(t),w(t))+\varepsilon C\int_{0}^{t}(w^{\prime},w).\label{ggg}%
\end{align}
Substituting the latter into relation (\ref{continuity}), choosing $t=T$, and
recalling that $\varepsilon w^{\prime}(T)=0$, we get%
\begin{align}
\varepsilon\int_{0}^{T}|w^{\prime}|^{2}+\frac{1}{2}|w(T)|^{2} &  \leq
C\int_{0}^{T}|f(v_{1})-f(v_{2})|^{2}\nonumber\\
&  ~~+\varepsilon C\int_{0}^{T}|(w^{\prime}(t),w(t))|+\varepsilon C\int
_{0}^{T}\int_{0}^{t}|(w^{\prime}(t),w(t))|\nonumber\\
&  \leq C\int_{0}^{T}|f(v_{1})-f(v_{2})|^{2}+\frac{\varepsilon}{2}\int_{0}%
^{T}|w^{\prime}|^{2}+\varepsilon C\int_{0}^{T}|w|^{2}.\label{bei}%
\end{align}
Integrating (\ref{continuity}) over $[0,T]$ and adding it to (\ref{bei}), we
obtain%
\begin{align*}
&  -\varepsilon\int_{0}^{T}(w^{\prime},w)+\varepsilon\int_{0}^{T}\int_{0}%
^{t}|w^{\prime}|^{2}+\frac{1}{2}\int_{0}^{T}|w|^{2}+\varepsilon\int_{0}%
^{T}|w^{\prime}|^{2}+\frac{|w(T)|^{2}}{2}\\
&  \leq C\int_{0}^{T}|f(v_{1})-f(v_{2})|^{2}+\frac{\varepsilon}{2}\int_{0}%
^{T}|w^{\prime}|^{2}+C\int_{0}^{T}|w|^{2}.
\end{align*}
By using once again estimate (\ref{ggg}), we conclude that
\begin{align*}
&  \frac{1-\varepsilon}{2}|w(T)|^{2}+\frac{1}{2}\int_{0}^{T}|w|^{2}%
+\frac{\varepsilon}{2}\int_{0}^{T}|w^{\prime}|^{2}\\
&  \leq C\int_{0}^{T}|f(v_{1})-f(v_{2})|^{2}+C\int_{0}^{T}|w|^{2}\\
&  \leq C\int_{0}^{T}|f(v_{1})-f(v_{2})|^{2}+\varepsilon C\int_{0}%
^{T}(w^{\prime}(t),w(t))\\
&  \text{ ~~}+\varepsilon C\int_{0}^{T}\int_{0}^{t}(w^{\prime}(t),w(t))\\
&  \leq C\int_{0}^{T}|f(v_{1})-f(v_{2})|^{2}+\frac{\varepsilon C}{2}%
|w(T)|^{2}\\
&  ~~+\varepsilon C\frac{1}{2}\int_{0}^{T}|w|^{2}.
\end{align*}
Thus, for $\varepsilon$ small enough, namely $\varepsilon\leq\min
\{(1+C)^{-1},(2C)^{-1}\}$, we have
\[
\frac{\varepsilon}{2}\int_{0}^{T}|w^{\prime}|^{2}+\frac{1}{4}\int_{0}%
^{T}|w|^{2}\leq C\int_{0}^{T}|f(v_{1})-f(v_{2})|^{2}.
\]

Since $f$ is continuous and linearly bounded, if $v_{1}-v_{2}\rightarrow0$ in
$L^{2}(0,T;H)$ then, $f(v_{1})-f(v_{2})\rightarrow0$ in $L^{2}(0,T;H)$ and
$w\rightarrow0$ in $H^{1}(0,T;H)$. This proves the continuity of
$S:L^{2}(0,T;H)\rightarrow H^{1}(0,T;H)$ and hence of $S:L^{2}%
(0,T;H)\rightarrow L^{2}(0,T;H)$.

\paragraph{Compactness.}

We now prove that the map $S$ is compact. Using the maximal regularity
estimate (\ref{maxRegEst}) and the linear growth of $f$, we get%

\begin{align*}
\varepsilon^{2}\left\Vert u^{\prime\prime}\right\Vert _{L^{2}(0,T;H)}%
^{2}+\left\Vert u^{\prime}\right\Vert _{L^{2}(0,T;H)}^{2}+\left\Vert
\xi\right\Vert _{L^{2}(0,T;H)}^{2}+\phi(u(T))  &  \leq C+\left\Vert
f(v)\right\Vert _{L^{2}(0,T;H)}^{2}\\
&  \leq C+C\left\Vert v\right\Vert _{L^{2}(0,T;H)}^{2}%
\end{align*}
Take now $v\in B$, where $B\subset L^{2}(0,T;H)$ is a bounded set. Then,
\[
\left\Vert u^{\prime}\right\Vert _{L^{2}\left(  0,T;H\right)  }^{2}\leq
C=C(B)
\]
and, recalling that $u(0)=u_{0}$, we have that$\ \left\Vert u\right\Vert
_{L^{2}\left(  0,T;H\right)  }^{2}\leq C.$ Testing equation
(\ref{Euler equation}), with $u^{\prime}$ and integrating first $[0,t]$ and
then on $[0,T]$, we get%
\begin{align*}
- \frac{\varepsilon}{2}\int_{0}^{T}|u^{\prime}|^{2}+\frac{\varepsilon T}{2}
|u^{\prime}(0)|^{2}+\int_{0}^{T}\int_{0}^{t}|u^{\prime}|^{2}-T\phi(u_{0}%
)+\int_{0}^{T}\phi(u)  &  \leq T\int_{0}^{T}|f(v)||u^{\prime}|\\
&  \leq C+C\left\Vert v\right\Vert _{L^{2}(0,T;H)}^{2}\text{.}%
\end{align*}
Thus, by using assumption (\ref{compact assumption}),%
\[
c_{X}\int_{0}^{T}|u|_{X}^{2}\leq C+\int_{0}^{T}\phi(u)\leq C.
\]
As $L^{2}(0,T;X)\cap H^{1}(0,T;H)$ is compactly embedded in $L^{2}(0,T;H)$, by
the Aubin-Lions Lemma \cite[Thm. 3]{Si}, the map $S$ is compact.

\paragraph{Boundedness of $A:= \{v\in L^{2}(0,T;H):v=\alpha S(v)$ for
$\alpha\in\lbrack0,1]\}.$}

In order to apply the Schaefer fixed-point Theorem \ref{schaefer} we are left
to prove that $A$ is bounded. First note that $A=\{0\}\cup\{v\in
L^{2}(0,T;H):v/\alpha=S(v)$ for $\alpha\in(0,1]\}$. Thus, $A$ is bounded if
and only if $\{v\in L^{2}(0,T;H):v/\alpha=S(v)$ for $\alpha\in(0,1]\}$ is
bounded. We now prove that $\tilde{A}:=\{u\in L^{2}(0,T;H):u=S(\alpha u)$ for
$\alpha\in(0,1]\}$ is bounded. This yields $A$ bounded.

Let $u\in\tilde{A}$. Then, there exists $\alpha\in(0,1]$ and $\xi\in
L^{2}(0,T;H)$ such that \thinspace$u$ solves
\begin{align*}
-\varepsilon u^{\prime\prime}+u^{\prime}+\xi-f(\alpha u) &  =0\text{
\ \ \ a.e. in }(0,T),\\
\xi &  \in\partial\phi(u)\text{ \ \ \ a.e. in }(0,T),\\
u^{\prime}(T) &  =0,\\
u(0) &  =u_{0}.
\end{align*}
Testing this equation with $u^{\prime}$ and integrating over $(0,t)$, we get
\begin{equation}
-\varepsilon\int_{0}^{t}(u^{\prime\prime},u^{\prime})+\int_{0}^{t}|u^{\prime
}|^{2}+\phi(u(t))-\phi(u_{0})=\int_{0}^{t}(f(\alpha u),u^{\prime
}).\label{10 bis}%
\end{equation}
Hence, recalling assumptions (\ref{compact assumption}) and
(\ref{linear growth f}), one has%
\begin{equation}
-\frac{\varepsilon}{2}|u^{\prime}(t)|^{2}+\frac{\varepsilon}{2}|u^{\prime
}(0)|^{2}+\frac{1}{2}\int_{0}^{t}|u^{\prime}|^{2}+\frac{c_{X}}{2}%
|u(t)|_{X}^{2}-\phi(u_{0})\leq C+C\alpha^{2}\int_{0}^{t}|u|^{2}\label{fff}%
\end{equation}
which, recalling that $\alpha\leq1$ and $|\cdot|\leq C|\cdot|_{X}$, yields%
\begin{equation}
\frac{1}{2}\int_{0}^{t}|u^{\prime}|^{2}+\frac{1}{2}|u(t)|^{2}\leq C+C\int
_{0}^{t}|u|^{2}+\frac{\varepsilon C}{2}|u^{\prime}(t)|^{2}.\label{cc}%
\end{equation}
Applying the Gronwall Lemma \ref{gronwall} from Appendix, we get%
\begin{align}
|u(t)|^{2} &  \leq C+\frac{\varepsilon C}{2}|u^{\prime}(t)|^{2}+C\int_{0}%
^{t}\left(  C+\frac{\varepsilon C}{2}|u^{\prime}(s)|^{2}\right)
\exp(C(t-s))\mathrm{{d}}s\nonumber\\
&  \leq C+\frac{\varepsilon C}{2}|u^{\prime}(t)|^{2}+\varepsilon C\exp
(TC)\int_{0}^{t}|u^{\prime}|^{2}\nonumber\\
&  \leq C+\frac{\varepsilon C}{2}|u^{\prime}(t)|^{2}+\varepsilon C\int_{0}%
^{t}|u^{\prime}|^{2}\text{.}\label{bb}%
\end{align}
Integrating (\ref{cc}) over $[0,T]$ and adding (\ref{cc}) to it along with the
choice $t=T$, one gets%
\begin{align*}
&  \frac{1}{2}\int_{0}^{T}\int_{0}^{t}|u^{\prime}|^{2}+\frac{1}{2}\int_{0}%
^{T}|u|^{2}+\frac{1}{2}\int_{0}^{T}|u^{\prime}|^{2}+\frac{1}{2}|u(T)|^{2}\\
&  \leq C+C\int_{0}^{T}\int_{0}^{t}|u|^{2}+\varepsilon C\int_{0}^{T}%
|u^{\prime}|^{2}+C\int_{0}^{T}|u|^{2}%
\end{align*}
and hence, thanks to estimate (\ref{bb}),%
\begin{align}
\frac{1}{2}\int_{0}^{T}|u|^{2}+\left(  \frac{1}{2}-C\varepsilon\right)
\int_{0}^{T}|u^{\prime}|^{2} &  \leq C+C(1{+}T)\int_{0}^{T}|u|^{2}\nonumber\\
&  \leq C+C\varepsilon(1{+}T)\int_{0}^{T}|u^{\prime}|^{2}\text{.}\nonumber
\end{align}
For all $\varepsilon$ small enough, we have that
\begin{equation}
\left\Vert u\right\Vert _{H^{1}(0,T;H)}\leq C\text{,}\label{final est}%
\end{equation}
where $C$ does not depend on $\varepsilon$ nor $\alpha$.

As a consequence of Theorem \ref{schaefer}, $S$ has a fixed point $u\in
L^{2}(0,T;H)$, $u=S\left(  u\right)  $. This solves the regularized equation
(\ref{true euler}) and
\[
u=\arg\min_{w\in K(u_{0})}\int_{0}^{T}\exp(-t/\varepsilon)\left(
\frac{\varepsilon}{2}|w^{\prime}|^{2}+\phi(w)-(f(u),w)\right)  \text{.}%
\]

\paragraph{The causal limit\label{causal limit}.}

The crucial issue for the WED theory is the so-called \textit{causal limit},
namely the convergence of the WED minimizers as $\varepsilon\rightarrow0$.

Let $u_{\varepsilon}$ be (one of) the solution(s) to the Euler-Lagrange
system. By testing regularized equation (\ref{true euler}) with
$u_{\varepsilon}^{\prime}$ and repeating the argument presented above with
$\alpha=1$ (cf. (\ref{10 bis}) and (\ref{final est})), we obtain that%
\[
\left\Vert u_{\varepsilon}\right\Vert _{H^{1}(0,T;H)}\leq C\text{,}%
\]
where $C$ does not depend on $\varepsilon$. Thanks to the maximal regularity
estimate (\ref{maxRegEst}) and of the sublinear growth of $f$, we have that
\begin{align}
&  \varepsilon^{2}\left\Vert u_{\varepsilon}^{\prime\prime}\right\Vert
_{L^{2}(0,T;H)}^{2}+\left\Vert u_{\varepsilon}^{\prime}\right\Vert
_{L^{2}(0,T;H)}^{2}+\left\Vert \xi_{\varepsilon}\right\Vert _{L^{2}%
(0,T;H)}^{2}+\phi(u_{\varepsilon}(T))\nonumber\\
&  \leq C+\left\Vert f(u_{\varepsilon})\right\Vert _{L^{2}(0,T;H)}%
^{2}\nonumber\\
&  \leq C\text{.}\label{final max reg}%
\end{align}
Furthermore, integrating (\ref{fff}) over $[0,T]$ (with $\alpha=1$), we deduce%
\[
\left\Vert u_{\varepsilon}\right\Vert _{L^{2}(0,T;X)}\leq C\text{.}%
\]
As a consequence of these uniform estimates and of the compact embedding
$H^{1}(0,T;H)\cap L^{2}(0,T;X)\hookrightarrow\hookrightarrow L^{2}(0,T;H)$
there exist (not relabeled) subsequences $u_{\varepsilon}$ and $\xi
_{\varepsilon}$ such that
\begin{align}
u_{\varepsilon} &  \rightarrow u\qquad\text{weakly in }H^{1}(0,T;H)\text{ and
in }L^{2}(0,T;X),\label{conv1}\\
u_{\varepsilon} &  \rightarrow u\qquad\text{strongly in }C([0,T];H)\text{ and
in }L^{2}(0,T;H),\label{conv2}\\
\xi_{\varepsilon} &  \rightarrow\xi\qquad\text{weakly in }L^{2}(0,T;H)\text{.
}\label{conv3}%
\end{align}
The demiclosedness of maximal monotone operators \cite{Br3} entails that
$\xi\in\partial\phi(u)$ a.e. in $(0,T)$. The trajectory $u$ is hence the
strong solution to (\ref{gradient flow2})-(\ref{gradient flow2.2}). This
concludes the proof of Theorem \ref{main teo}.

\paragraph{Explicit convergence rate.}

Under additional assumptions on $f$, we can obtain an estimate for the
convergence rate of $\left\Vert u_{\varepsilon}-u \right\Vert _{C([0,T];H)}$.
In particular, let $f_{i}:H\rightarrow H,$ $i=1,2,$ be such that
\begin{align}
&  f =f_{1}+f_{2},\nonumber\\
&  f_{1}\text{ is Lipschitz continuous and }\nonumber\\
&  -f_{2}\text{ is monotone.} \label{add ass f}%
\end{align}

By testing the difference between the regularized equation and the gradient
flow equation with $w=u-u_{\varepsilon}$ and using the convexity of $\phi$, we
get (cf. \cite{Mi-St})
\begin{align}
\frac{\varepsilon}{2}\int_{0}^{t}|w^{\prime}|^{2}+\frac{1}{4}|w(t)|^{2}  &
\leq|u_{0}-u_{0\varepsilon}|^{2}+\varepsilon\int_{0}^{t}|u^{\prime}%
|^{2}+\varepsilon^{2}|u_{\varepsilon}^{\prime}(t)|^{2}\nonumber\\
&  +\frac{\varepsilon^{2}}{2}|u_{\varepsilon}^{\prime}(0)|^{2}+\int_{0}%
^{t}(f(u)-f(u_{\varepsilon}),u-u_{\varepsilon})\nonumber\\
&  \leq C\varepsilon+\int_{0}^{t}(f(u)-f(u_{\varepsilon}),u-u_{\varepsilon
})\nonumber\\
&  \leq C\varepsilon+L\int_{0}^{t}|w|^{2}\text{,} \label{strong conv 1}%
\end{align}
where $L$ is the Lipschitz constant of $f_{1}$. Applying the Gronwall Lemma,
we obtain%
\begin{equation}
\frac{\varepsilon}{2}\int_{0}^{t}|w^{\prime}|^{2}+|w(t)|^{2}\leq
C\varepsilon\label{strong conv2}%
\end{equation}
which entails
\[
\left\Vert u_{\varepsilon}-u\right\Vert _{C([0,T];H)}\leq C\varepsilon^{1/2}.
\]

\section{Nonconvex energies\label{nonconvex energy}}

We now come to the proof of Theorem \ref{full gen th}, namely we consider
$\phi=\varphi_{1}-\varphi_{2}$ nonconvex. This forces us to introduce a
further approximation which will be then removed before taking the causal
limit $\varepsilon\rightarrow0$. In particular, we regularize the problem for
all $\lambda>0$ by replacing $\varphi_{2}$ with its \textit{Moreau-Yosida
regularization} \cite{Br3}:
\[
\varphi_{2}^{\lambda}(u)=\inf_{v\in H}\left(  \frac{1}{2\lambda}%
|u-v|^{2}+\varphi_{2}(v)\right)  =\frac{1}{2\lambda}|u-J_{\lambda}%
u|^{2}+\varphi_{2}(J_{\lambda}u)\qquad\text{for all }u\in H\text{,}%
\]
where $J_{\lambda}u$ denotes the \textit{resolvent} for $\partial\varphi_{2}$
It is well known that
\begin{align}
\varphi_{2}^{\lambda} &  \in C^{1,1}\text{,}\\
\varphi_{2}^{\lambda}(u) &  \leq\varphi_{2}(u)\qquad\text{for all }u\in
D\left(  \varphi_{2}\right)  \text{,}\label{comparison yosida}\\
\mathrm{{D}}\varphi_{2}^{\lambda}(u) &  =\partial\varphi_{2}(J_{\lambda
}u)\qquad\text{for all }u\in H\text{,}\label{resolvent}\\
|\mathrm{{D}}\varphi_{2}^{\lambda}(u)| &  \leq|\eta|\qquad\text{for all
}[u,\eta]\in\partial\varphi_{2}\text{.}\label{norm yosida}%
\end{align}
Here $\mathrm{{D}}\varphi_{2}^{\lambda}$ denotes the Fr\'{e}chet derivative of
$\varphi_{2}^{\lambda}$. In particular, $\mathrm{{D}}\varphi_{2}^{\lambda
}:H\rightarrow H$ is (single-valued and) Lipschitz continuous. Hence,
$g=f+\mathrm{{D}}\varphi_{2}^{\lambda}$ satisfies assumption
(\ref{linear growth f}). Thus, Theorem \ref{main teo} ensures the existence of
(at least) a solution $u_{\varepsilon,\lambda}$ to
\begin{align}
-\varepsilon u_{\varepsilon,\lambda}^{\prime\prime}+u_{\varepsilon,\lambda
}^{\prime}+\xi_{\varepsilon,\lambda} &  =f(u_{\varepsilon,\lambda
})+\mathrm{{D}}\varphi_{2}^{\lambda}(u_{\varepsilon,\lambda})\qquad\text{a.e.
in }(0,T),\label{Euler Yosida}\\
\xi_{\varepsilon,\lambda} &  \in\partial\varphi_{1}(u_{\varepsilon,\lambda
})\qquad\text{a.e. in }(0,T),\\
u_{\varepsilon,\lambda}(0) &  =u_{0},\qquad u_{\varepsilon,\lambda}^{\prime
}(T)=0.
\end{align}

We now derive estimates on $u_{\varepsilon,\lambda}$ which are uniform with
respect to $\lambda$ (as well as $\varepsilon$) in order to pass to the limit.
Henceforth the symbol $C$ will be independent of $\lambda$ as well. Testing
(\ref{Euler Yosida}) by $u_{\varepsilon,\lambda}^{\prime}$ and integrating
over $\left[  0,t\right]  $, we obtain
\[
-\varepsilon\int_{0}^{t}(u_{\varepsilon,\lambda}^{\prime\prime},u_{\varepsilon
,\lambda}^{\prime})+\int_{0}^{t}|u_{\varepsilon,\lambda}^{\prime}|^{2}%
+\varphi_{1}(u_{\varepsilon,\lambda}(t))-\varphi_{2}^{\lambda}(u_{\varepsilon
,\lambda}(t))=\int_{0}^{t}(f(u_{\varepsilon,\lambda}),u_{\varepsilon,\lambda
}^{\prime})+\varphi_{1}(u_{0})-\varphi_{2}^{\lambda}(u_{0}).
\]
Hence, using assumption (\ref{controll 1}) and inequality
(\ref{comparison yosida}), we get%
\begin{equation}
-\varepsilon\int_{0}^{t}(u_{\varepsilon,\lambda}^{\prime\prime},u_{\varepsilon
,\lambda}^{\prime})+\int_{0}^{t}|u_{\varepsilon,\lambda}^{\prime}|^{2}%
+(1{-}k_{1})\varphi_{1}(u_{\varepsilon,\lambda}(t))\leq\int_{0}^{t}%
(f(u_{\varepsilon,\lambda}),u_{\varepsilon,\lambda}^{\prime})+C\text{.}%
\label{tarzan}%
\end{equation}
Arguing as in Section \ref{causal limit}, we obtain
\begin{equation}
\left\Vert u_{\varepsilon,\lambda}\right\Vert _{H^{1}(0,T;H)}\leq
C\text{.}\label{quattro otto}%
\end{equation}
Integrating relation (\ref{tarzan}) over $[0,T]$ and using the above
estimates, we additionally find
\begin{equation}
\int_{0}^{T}\varphi_{1}(u_{\varepsilon,\lambda})\leq C\text{.}\label{est}%
\end{equation}
Thus, thanks to assumption (\ref{compact assumption}), we deduce%
\begin{align*}
\left\Vert u_{\varepsilon,\lambda}\right\Vert _{L^{2}\left(  0,T;X\right)  }
& \leq C\text{,}\\
\left\Vert u_{\varepsilon,\lambda}\right\Vert _{C\left(  [0,T];H\right)  }  &
\leq C\text{.}%
\end{align*}
Applying, the maximal regularity estimate (\ref{maxRegEst}), we have
\begin{align*}
& \varepsilon^{2}\left\Vert u_{\varepsilon,\lambda}^{\prime\prime}\right\Vert
_{L^{2}(0,T,H)}^{2}+\left\Vert u_{\varepsilon,\lambda}^{\prime}\right\Vert
_{L^{2}(0,T,H)}^{2}+\left\Vert \xi_{\varepsilon,\lambda}\right\Vert
_{L^{2}(0,T;H)}^{2}+\varphi_{1}(u_{\varepsilon,\lambda}(T))\\
& \leq\left\Vert f(u_{\varepsilon,\lambda})+\mathrm{{D}}\varphi_{2}^{\lambda
}(u_{\varepsilon,\lambda})\right\Vert _{L^{2}(0,T;H)}^{2}\text{.}%
\end{align*}
Thanks to (\ref{norm yosida}) and assumption (\ref{controll 2}), we estimate
\[
|f(u_{\varepsilon,\lambda})+\mathrm{{D}}\varphi_{2}^{\lambda}(u_{\varepsilon
,\lambda})|^{2}\leq(1+\delta)k_{2}|\xi_{\varepsilon,\lambda}|^{2}+C_{\delta
}|f(u_{\varepsilon,\lambda})|^{2}%
\]
for every $\delta>0$ and some $C_{\delta}$. Thus, as $k_{2}<1$, choosing
$\delta$ sufficiently small,
\begin{align}
\varepsilon^{2}\left\Vert u_{\varepsilon,\lambda}^{\prime\prime}\right\Vert
_{L^{2}(0,T,H)}^{2}+\left\Vert u_{\varepsilon,\lambda}^{\prime}\right\Vert
_{L^{2}(0,T,H)}^{2}+\left\Vert \xi_{\varepsilon,\lambda}\right\Vert
_{L^{2}(0,T;H)}^{2}+\varphi_{1}(u_{\varepsilon,\lambda}(T)) &  \leq C\text{,
}\label{es}\\
\varepsilon\left\Vert u_{\varepsilon,\lambda}^{\prime}\right\Vert _{C\left(
[0,T];H\right)  }^{2} &  \leq C\text{.}\label{es2}%
\end{align}
As a consequence of assumptions (\ref{controll 1}), (\ref{controll 2}) and of
the above estimates, we get
\[
\int_{0}^{T}\varphi_{2}(u_{\varepsilon,\lambda})+\int_{0}^{T}|\mathrm{{D}%
}\varphi_{2}^{\lambda}(u_{\varepsilon,\lambda})|^{2}\leq C\text{.}%
\]

We deal first with the passage to the limit for $\lambda\rightarrow0$ for
$\varepsilon$ fixed. Owing to the obtained uniform estimates, up to some not
relabeled subsequence, we have%
\begin{align*}
u_{\varepsilon,\lambda} &  \rightarrow u_{\varepsilon}\qquad\text{weakly in
}H^{2}(0,T;H)\text{ and in }L^{2}(0,T;X),\\
\xi_{\varepsilon,\lambda} &  \rightarrow\xi_{\varepsilon}\qquad\text{weakly in
}L^{2}(0,T;H)\text{,}\\
\mathrm{{D}}\varphi_{2}^{\lambda}(u_{\varepsilon,\lambda}) &  \rightarrow
\eta_{\varepsilon}\qquad\text{weakly in }L^{2}(0,T;H)\text{,}\\
u_{\varepsilon,\lambda} &  \rightarrow u_{\varepsilon}\qquad\text{strongly in
}C([0,T];H).
\end{align*}
Using continuity of $f$, we obtain
\[
f(u_{\varepsilon,\lambda})\rightarrow f(u_{\varepsilon})\qquad\text{strongly
in }L^{2}(0,T;H).
\]
As a consequence of the demiclosedness of maximal monotone operators we
conclude that $\xi_{\varepsilon}\in\partial\varphi_{1}(u_{\varepsilon})$
almost everywhere. The inclusion $\eta_{\varepsilon}\in\partial\varphi
_{2}(u_{\varepsilon})$ follows then by the standard monotonicity argument
\cite[Sec. 1.2]{Ba}. As $H=H^{\ast}$ is compactly embedded in $X^{\ast}$ we
have the following convergence result (again for a not-relabeled subsequence)%
\[
u_{\varepsilon,\lambda}^{\prime}\rightarrow u_{\varepsilon}^{\prime}%
\qquad\text{strongly in }C([0,T];X^{\ast}).
\]
In particular, $u_{\varepsilon}^{\prime}(T)=0$ and $u_{\varepsilon}$ solves
equation (\ref{euler eq 2}). In addition, the sequence $u_{\varepsilon}$
satisfies the estimates (\ref{quattro otto})-(\ref{es2}) and
\[
\int_{0}^{T}\varphi_{2}(u_{\varepsilon})+\int_{0}^{T}|\eta_{\varepsilon}%
|^{2}\leq C.
\]

Let us now consider the causal limit $\varepsilon\rightarrow0$. By taking (not
relabeled) subsequences one has
\begin{align*}
u_{\varepsilon} &  \rightarrow u\qquad\text{weakly in }H^{1}(0,T;H)\text{ and
in }L^{2}(0,T;X),\\
u_{\varepsilon} &  \rightarrow u\qquad\text{strongly in }C\left(
[0,T];H\right)  ,\\
\xi_{\varepsilon} &  \rightarrow\xi\qquad\text{weakly in }L^{2}(0,T;H)\text{,
}\\
\eta_{\varepsilon} &  \rightarrow\eta\qquad\text{weakly in }L^{2}%
(0,T;H)\text{,}\\
f(u_{\varepsilon}) &  \rightarrow f(u)\qquad\text{strongly in }L^{2}(0,T;H).
\end{align*}
By the demiclosedness of maximal monotone operators, one concludes $\xi
(t)\in\partial\varphi_{1}(u(t))$ and $\eta(t)\in\partial\varphi_{2}(u(t))$ for
almost every $t\in\left(  0,T\right)  $. Hence, $u$ solves equation
(\ref{non convex perturbed gf}) and the assertion of Theorem \ref{full gen th} follows.

\section{More general initial data \label{generalization}}

The results of Theorem \ref{full gen th} are also valid under weaker
assumptions on the initial datum $u_{0}$. Aiming at clarity, we first
illustrate the case of a convex energy. We use here the notation of Section
\ref{lambda convex energy} and follow closely the argument in \cite[Secs.
2.5-6]{Mi-St}. From \cite{Br1,Br2} we define the \textit{interpolation} set
$D_{r,p}$ as%
\[
D_{r,p}=\left\{  u\in\overline{D(\partial\phi)}:\varepsilon\longmapsto
\varepsilon^{-r}|u-J_{\varepsilon}u|\in L^{p}\left(  0,1,\varepsilon
^{-1}\mathrm{d}\varepsilon\right)  \right\}  \text{,}%
\]
where $J_{\varepsilon}=(id+\varepsilon\partial\phi)^{-1}$ is the standard
\textit{resolvent} operator. We recall the following properties from
\cite{Br1,Br2}%

\begin{align*}
&  u_{0}\in D_{r,p} \text{ iff }\exists\varepsilon\in\lbrack0,1]\longmapsto
v(\varepsilon):v\in W_{\mathrm{{loc}}}^{1,1}(0,1]\text{, continuous in }[0,1],
\quad v(0)=u_{0},\\
&  v(\varepsilon)\in D(\partial\phi)\text{ a.e. and }\varepsilon
^{1-r}(|(\partial\phi(v\left(  \varepsilon\right)  ))^{\circ}|+|v^{\prime
}\left(  \varepsilon\right)  |)\in L^{p}(0,1,\varepsilon^{-1}\mathrm{d}%
\varepsilon).\\
\\
&  D(\partial\phi)\subset D(\phi)=D_{1/2,2}\subset D_{1/2,\infty}\subset
D_{r,\infty}\text{ for } r\in(0,1/2).
\end{align*}

Let now $u_{0}\in D_{r,\infty}$ for $r\in(0,1/2]$ and the sequence
$u_{0\varepsilon}\in D(\partial\phi)$ be such that $u_{0\varepsilon
}\rightarrow u_{0}$ strongly in $H$ and
\[
\varepsilon^{-r}|u_{0\varepsilon}-u_{0}|+\varepsilon^{1-r}|(\partial
\phi(u_{0\varepsilon}))^{\circ}|\leq C.
\]
Arguing as in Section \ref{lambda convex energy} it is possible to prove
existence of a solution $u_{\varepsilon}$ to the regularized problem
\begin{align*}
-\varepsilon u_{\varepsilon}^{\prime\prime}+u_{\varepsilon}^{\prime}%
+\xi_{\varepsilon}-f(u_{\varepsilon})  &  =0\text{ \ \ \ \ a.e. in }(0,T),\\
\xi_{\varepsilon}  &  \in\partial\phi(u_{\varepsilon})\text{ \ \ \ \ a.e. in
}(0,T),\\
u_{\varepsilon}^{\prime}(T)  &  =0,\\
u_{\varepsilon}(0)  &  =u_{0\varepsilon}.
\end{align*}
Estimate (\ref{final max reg}) reads in this case
\[
\varepsilon^{2}\left\Vert u_{\varepsilon}^{\prime\prime}\right\Vert
_{L^{2}(0,T;H)}^{2}+\left\Vert u_{\varepsilon}^{\prime}\right\Vert
_{L^{2}(0,T;H)}^{2}+\left\Vert \xi_{\varepsilon}\right\Vert _{L^{2}%
(0,T;H)}^{2}+\phi(u_{\varepsilon}(T))\leq C+\left\Vert f(u_{\varepsilon
})\right\Vert ^{2}\leq C\varepsilon^{2r-1}\text{.}%
\]
If $u_{0}\in D(\phi)=D_{1/2,2}$ then $r=1/2$ and the estimate suffices to pass
to the limit. By assuming (\ref{add ass f}), we can argue as in
(\ref{strong conv 1})-(\ref{strong conv2}) and obtain%
\begin{align*}
\frac{\varepsilon}{2}\int_{0}^{t}|u^{\prime}-u_{\varepsilon}^{\prime}%
|^{2}+\frac{1}{4}|u(t)-u_{\varepsilon}(t)|^{2}  &  \leq C\left(
|u_{0}-u_{0\varepsilon}|^{2}+\varepsilon\int_{0}^{t}|u^{\prime}|^{2}%
+\varepsilon^{2}|u_{\varepsilon}^{\prime}(t)|^{2}+\frac{\varepsilon^{2}}%
{2}|u_{\varepsilon}^{\prime}(0)|^{2}\right) \\
&  \leq C\varepsilon^{2r}\text{.}%
\end{align*}
Thus, uniform convergence holds for all $r\in(0,1/2]$.

We deal now with the nonconvex energy case. Let $\varphi_{1},~\varphi_{2},~f$
satisfy assumptions of Theorem \ref{full gen th}, define%
\begin{align*}
D_{r,p}(\varphi_{1})  &  =\{u\in\overline{D(\partial\varphi_{1})}%
:\varepsilon\longmapsto\varepsilon^{-r}|u-J_{\varepsilon}u|\in L^{p}%
(0,1,\varepsilon^{-1}\mathrm{d}\varepsilon)\}\text{,}\\
J_{\varepsilon}  &  =(id+\varepsilon\partial\varphi_{1})^{-1}.
\end{align*}
Assume $u_{0} \in D_{r, \infty}(\varphi_{1})$ for $r \in(0,1/2]$ so that there
exists a sequence $u_{0\varepsilon}\in D(\partial\varphi_{1})\subset
D(\partial\varphi_{2})$ such that $u_{0\varepsilon} \to u_{0}$ strongly in
$H$,
\[
\varepsilon^{-r}|u_{0\varepsilon}-u_{0}|+\varepsilon^{1-r}|(\partial
\varphi_{1}(u_{0\varepsilon}))^{\circ}|\leq C,
\]
and $\varepsilon^{1-r}|\partial\varphi_{2}(u_{0\varepsilon})|\leq C$ (by using
assumption (\ref{controll 2})). This is enough to combine the uniform
estimates of Section \ref{nonconvex energy} with the approximation of the
initial datum $u_{0\varepsilon}\in D_{r,p}(\varphi_{1})$ and extend the
results of Theorem \ref{full gen th} to the case $u_{0} \in D_{r,\infty
}(\varphi_{1})$.

\section{Applications \label{applications}}

Our results yield a generalization to the nonpotential perturbation case of
the theory in \cite{Ak-St,Mi-St}. Our analysis applies to most of the examples
described in Section 6 of \cite{Ak-St} and Section 7 of \cite{Mi-St}, e.g.,
quasilinear parabolic PDEs, the Allen-Cahn equation, the sublinear heat
equation. Moreover, the occurrence of a nonpotential term allows us to apply
the abstract theory to \textit{systems}, in particular to reaction-diffusion
and nonlinear diffusion systems.

\subsection{Reaction-diffusion systems}

Consider the system%
\begin{align}
u_{t}  &  =D_{1}\Delta u+f_{1}(u,v) \qquad\text{in }\Omega\times
(0,T)\text{,}\label{reaction diffusion system}\\
v_{t}  &  =D_{2}\Delta v+f_{2}(u,v) \qquad\text{in }\Omega\times
(0,T)\text{,}\\
\partial_{n} u  &  =0 = \partial_{n}v \qquad\text{on }\partial\Omega
\times(0,T)\text{,} \label{reaction diffusion system2}%
\end{align}
where $\Omega$ is a bounded subset of $%
%TCIMACRO{\U{211d} }%
%BeginExpansion
\mathbb{R}
%EndExpansion
^{d}$ with sufficiently smooth boundary $\partial\Omega$ and $n$ denotes the
unit outward normal vector on $\partial\Omega$ . Assume $D_{1},D_{2}>0$ and
\[
f%
\begin{pmatrix}
u\\
v
\end{pmatrix}
=%
\begin{pmatrix}
f_{1}(u,v)\\
f_{2}(u,v)
\end{pmatrix}
:%
%TCIMACRO{\U{211d} }%
%BeginExpansion
\mathbb{R}
%EndExpansion
^{2}\rightarrow%
%TCIMACRO{\U{211d} }%
%BeginExpansion
\mathbb{R}
%EndExpansion
^{2}%
\]
to be a linearly bounded continuous function. System
(\ref{reaction diffusion system})-(\ref{reaction diffusion system2}) arises in
a variety of different situations. The choice%
\begin{align}
f_{1}(u,v)  &  =Au\left(  1-\frac{u}{K}\right)  -\frac{Buv}{1+Eu}%
,\label{lotka volterra}\\
f_{2}(u,v)  &  =\frac{Cuv}{1+Eu}-Dv, \label{lokta volterra 2}%
\end{align}
where $K>0$ and $A,B,C,D,E\geq0$ models a \textit{diffusive prey-predator
system} (cf., e.g., \cite{Mu,Du1,Du2}). In this contest, $u$ represents the
number of preys and $v$ the number of predators, $K$ is the so-called
\textit{capacity} of the environment for the prey and $D_{1}>0$ and $D_{2}>0$
are the corresponding diffusion coefficients. Usually one is interested in
solutions $\left(  u,v\right)  $ such that $u\in\lbrack0,K]$ and $v>0$. This
implies that (\ref{lotka volterra})-(\ref{lokta volterra 2}) can be
equivalently rewritten as%
\begin{align}
f_{1}(u,v)  &  =AU\left(  1-\frac{U}{K}\right)  -\frac{BUV}{1+EU},\label{a}\\
f_{2}(u,v)  &  =\frac{CUV}{1+EU}-DV \label{b}%
\end{align}
where%
\[
U=\left(  \min\{u,K\}\right)  ^{+},\text{ }V=(v)^{+}.
\]

By choosing a different form of $f$, the system relates to \textit{pattern
formation} in animal coating (cf. \cite{Mu,Mu2}). The reaction term takes here
the form
\begin{align}
f_{1}(u,v)  &  =\alpha-u-h(u,v),\label{c}\\
f_{2}(u,v)  &  =\gamma(\beta-v)-h(u,v),\label{d}\\
h(u,v)  &  =\frac{\rho uv}{1+u+\delta u^{2}} \label{ultima}%
\end{align}
where $\alpha$, $\beta$, $\gamma$, $\delta$, and $\rho$ are positive
constants. As $u$ and $v$ represent concentrations and $u,v>0$, we can
conveniently rewrite (\ref{ultima}) as
\begin{align}
h(u,v)=\frac{\rho uv}{1+(u)^{+}+\delta u^{2}} . \label{new h}%
\end{align}

Yet another example of choice of $f$ of application interest is
\begin{align}
f_{1}(u,v)  &  =p-ug(v),\label{e}\\
f_{2}(u,v)  &  =k(ug(v)-v), \label{f}%
\end{align}
which is related to \textit{combustion}. Here $p$ and $k$ are positive
constants and $g(v)=\exp(v/(1+\delta v))$ with $\delta>0$. The latter choices
are known as Scott-Wang-Showalter model \cite{SWS}. Here $u$ denotes the
concentration of an intermediate chemical species and $v$ is the temperature.

Note that the reaction terms corresponding to any of the choices
(\ref{a})-(\ref{b}), (\ref{c})-(\ref{d}) together with (\ref{new h}), or
(\ref{e})-(\ref{f}) are continuous and satisfy assumption
(\ref{linear growth f}). We are hence in the position of applying our abstract
theory to all these systems.

\bigskip

At first we rewrite system (\ref{reaction diffusion system}%
)-(\ref{reaction diffusion system2}) as%
\begin{equation}%
\begin{pmatrix}
u_{t}\\
v_{t}%
\end{pmatrix}
+\partial\phi%
\begin{pmatrix}
u\\
v
\end{pmatrix}
\ni\tilde{f}%
\begin{pmatrix}
u\\
v
\end{pmatrix}
\qquad\text{in }(0,T), \label{gf reaction diffusion}%
\end{equation}
where
\begin{align*}
\phi%
\begin{pmatrix}
u\\
v
\end{pmatrix}
=\left\{
\begin{array}
[c]{cc}%
\displaystyle{ \frac{1}{2}\int_{\Omega}D_{1}|\nabla u|^{2}+D_{2}|\nabla
v|^{2}+|u|^{2}+|v|^{2} } & \text{if }u\in D\text{ and }v\in D,\\
+\infty & \text{else,}%
\end{array}
\right. \\
D=\{u\in H^{2}(\Omega): \partial_{n} u=0\text{ on }\partial\Omega\}, \qquad
H=L^{2}(\Omega), \qquad X=H^{1}(\Omega),\\
\tilde{f}%
\begin{pmatrix}
u\\
v
\end{pmatrix}
=f%
\begin{pmatrix}
u\\
v
\end{pmatrix}
+%
\begin{pmatrix}
u\\
v
\end{pmatrix}
.
\end{align*}
It is straightforward to check that $\phi$ and $\tilde{f}$ satisfy the
assumptions of Theorem \ref{full gen th}. We hence have the following.

\begin{theorem}
Let $u_{0},v_{0}\in D$. Then, for every $T>0$ and for $\varepsilon
=\varepsilon(T)>0$ sufficiently small, the system%

\begin{align*}
-\varepsilon u_{tt}+u_{t} &  =D_{1}\Delta u+f_{1}(u,v)\qquad\text{ in }%
\Omega\times(0,T)\text{,}\\
-\varepsilon v_{tt}+v_{t} &  =D_{2}\Delta v+f_{2}(u,v)\qquad\text{ in }%
\Omega\times(0,T)\text{,}\\
\partial_{n}u &  =\partial_{n}v=0\qquad\ \text{on }\partial\Omega
\times(0,T)\text{,}\\
u(0) &  =u_{0},\quad v(0)=v_{0}\qquad\text{ in }\Omega,\\
\varepsilon u^{\prime}(T) &  =0\text{, }\varepsilon v^{\prime}(T)=0
\end{align*}

admits at least a solution $\left(  u_{\varepsilon},v_{\varepsilon}\right)
\in H^{2}(0,T;(L^{2}(\Omega))^{2}) \cap L^{2}(0,T;(H^{1}(\Omega))^{2}). $
Moreover, $u_{\varepsilon}\rightarrow u$ and $v_{\varepsilon}\rightarrow v$
weakly in $H^{1}\left(  0,T;L^{2}(\Omega)\right)  $ and strongly in $C\left(
[0,T];L^{2}(\Omega)\right)  $ where $\left(  u,v\right)  $ is a solution to
system \emph{(\ref{reaction diffusion system}%
)-(\ref{reaction diffusion system2})}.
\end{theorem}

\subsection{Nonlinear diffusion}

We can also apply our abstract results to systems of nonlinear
reaction-diffusion equations of the following type%
\begin{align}
u_{t} &  =D_{1}\Delta_{p}u+|u|^{m-2}u-|u|^{q-2}u+f_{1}(u,v)\qquad\text{ in
}\Omega\times(0,T)\text{,}\label{nonlin ex}\\
v_{t} &  =D_{2}\Delta_{p}v+|v|^{m-2}v-|v|^{q-2}v+f_{2}(u,v)\qquad\text{ in
}\Omega\times(0,T)\text{,}\\
\partial_{n}u &  =\partial_{n}v=0\qquad\text{ on }\partial\Omega
\times(0,T)\text{,}\label{nonlin ex2}%
\end{align}
where $1<q<m<+\infty$, $1<p<+\infty$, and $\Delta_{p}$ is the so-called
$p$\textit{-Laplacian} given by%
\[
\Delta_{p}u=\nabla\cdot(|\nabla u|^{p-2}\nabla u).
\]
In order to write system (\ref{nonlin ex})-(\ref{nonlin ex2}) to the abstract
setting, we define $H=L^{2}(\Omega)$, $X=D(\varphi_{1})=W^{1,p}(\Omega)\cap
L^{m}(\Omega)$, $D(\varphi_{2})=L^{q}(\Omega)$,
\[
\varphi_{1}%
\begin{pmatrix}
u\\
v
\end{pmatrix}
=\left\{
\begin{array}
[c]{cc}%
\displaystyle{\int_{\Omega}\frac{D_{1}}{p}|\nabla u|^{p}+\frac{D_{2}}%
{p}|\nabla v|^{p}+\frac{1}{m}|u|^{m}+\frac{1}{m}|v|^{m}} & \text{if }u\in
D(\varphi_{1})\text{ and }v\in D(\varphi_{1}),\\
+\infty & \text{else},
\end{array}
\right.
\]
and
\[
\varphi_{2}%
\begin{pmatrix}
u\\
v
\end{pmatrix}
=\left\{
\begin{array}
[c]{cc}%
\displaystyle{\frac{1}{q}\int_{\Omega}|u|^{q}+|v|^{q}} & \text{if }u\in
D(\varphi_{2})\text{ and }v\in D(\varphi_{2}),\\
+\infty & \text{else.}%
\end{array}
\right.
\]
Moreover, we assume
\[
f%
\begin{pmatrix}
u\\
v
\end{pmatrix}
=%
\begin{pmatrix}
f_{1}(u,v)\\
f_{2}(u,v)
\end{pmatrix}
:%
%TCIMACRO{\U{211d} }%
%BeginExpansion
\mathbb{R}
%EndExpansion
^{2}\rightarrow%
%TCIMACRO{\U{211d} }%
%BeginExpansion
\mathbb{R}
%EndExpansion
^{2}%
\]
to be linearly bounded and continuous. It can be easily checked that
assumptions of Theorem \ref{full gen th} are satisfied (cf. Section 6.1 of
\cite{Ak-St}) and we hence conclude the following.

\begin{theorem}
Let $u_{0},v_{0}\in D(\partial\varphi_{1})$. Then, for every $T>0$ and for
$\varepsilon=\varepsilon(T)>0$ sufficiently small, the system
\begin{align*}
-\varepsilon u_{tt}+u_{t} &  =D_{1}\Delta_{p}u+|u|^{m-2}u-|u|^{q-2}%
u+f_{1}(u,v)\text{ \ \ \ in }\Omega\times(0,T)\text{,}\\
-\varepsilon v_{tt}+v_{t} &  =D_{2}\Delta_{p}v+|v|^{m-2}v-|v|^{q-2}%
v+f_{2}(u,v)\text{ \ \ \ in }\Omega\times(0,T)\text{,}\\
\partial_{n}u &  =\partial_{n}v=0\text{ \ \ \ on }\partial\Omega
\times(0,T)\text{,}\\
u(0) &  =u_{0}\text{, \ \ }v(0)=v_{0}\text{\ \ \ in }\Omega\\
\varepsilon u^{\prime}(T) &  =0\text{, \ \ \ }\varepsilon v^{\prime}(T)=0
\end{align*}
admits at least a solution
\[
\left(  u_{\varepsilon},v_{\varepsilon}\right)  \in H^{2}(0,T;(L^{2}%
(\Omega))^{2})\cap L^{p}(0,T;(W^{1,p}(\Omega))^{2})\cap L^{m}(0,T;(L^{m}%
(\Omega))^{2}).
\]
Moreover, $u_{\varepsilon}\rightarrow u$ and $v_{\varepsilon}\rightarrow v$
weakly in $H^{1}\left(  0,T;L^{2}(\Omega)\right)  $ and strongly in $C\left(
[0,T];L^{2}(\Omega)\right)  $ where $\left(  u,v\right)  $ is a solution to
system \emph{(\ref{nonlin ex})-(\ref{nonlin ex2})}.
\end{theorem}

\section{Appendix}

We collect here two tools for the Reader's convenience.

\begin{lemma}
[Gronwall lemma]\label{gronwall}Let $\alpha,u\in L^{1}(0,T)$ and $B>0$.
Assume
\begin{equation}
u(t)\leq\alpha(t)+\int_{0}^{t}Bu(s)\mathrm{d}s~ \qquad\text{for a.e. } t
\in(0,T). \label{quasi last}%
\end{equation}
Then,
\begin{equation}
u(t)\leq\alpha(t)+\int_{0}^{t}B\alpha(s)\exp(B(t-s))\mathrm{d}s. \label{gw}%
\end{equation}

\end{lemma}

\begin{proof}
Define $v(t)= \exp(-Bt)\int_{0}^{t}Bu(s)\mathrm{d}s$. Then, $v\in
W^{1,1}(0,T),$ $v(0)=0$ and
\[
v^{\prime}(t)=B \exp(-Bt)\left(  u(t)-\int_{0}^{t}Bu(s)\mathrm{d}s\right)
\leq B \exp(-Bt)\alpha(t) \qquad\text{for }a.a.~t\in(0,T).
\]
Thus, by integrating over $(0,t)$ we get
\[
\exp(-Bt)\int_{0}^{t}Bu(s)\mathrm{d}s=v(t)\leq\int_{0}^{t}B \exp
(-Bs)\alpha(s)\mathrm{d}s
\]
yielding
\begin{equation}
\int_{0}^{t}Bu(s)\mathrm{d}s\leq\int_{0}^{t}B \exp(B(t-s))\alpha
(s)\mathrm{d}s\text{.} \label{last}%
\end{equation}
By substituting (\ref{last}) into (\ref{quasi last}) we get (\ref{gw}).
\end{proof}

\begin{theorem}
[{Schaefer fixed-point Theorem \cite[Thm. 4, Ch. 9]{Ev}}]\label{schaefer}Let
$X$ be a Banach space, $S:X\rightarrow X$ be continuous and compact, and
\[%
%TCIMACRO{\dbigcup \limits_{\alpha\in\lbrack0,1]}}%
%BeginExpansion
{\displaystyle\bigcup\limits_{\alpha\in\lbrack0,1]}}
%EndExpansion
\{u\in X:u=\alpha S(u)\}
\]
be bounded. Then, $S$ has a fixed point.
\end{theorem}

%\}$ be bounded. Then, $\ S$ has a fixed point.


\begin{thebibliography}{999999}                                                                                           %


\bibitem[Ak-St]{Ak-St}G. Akagi, U. Stefanelli. \textit{A variational principle
for gradient flows of nonconvex energies,} J. Convex Anal. (2015), to appear.

\bibitem[Ak-St2]{Ak-St2}G. Akagi, U. Stefanelli. \textit{Weighted
energy-dissipation functionals for doubly nonlinear evolution,} J. Funct.
Anal. 260\ (2011), 2541-2578.

\bibitem[Ak-St3]{Ak-St3}G. Akagi, U. Stefanelli. \textit{Doubly nonlinear
evolution equations as convex minimization,} SIAM J. Math. Anal. 46 (2014), 1922-1945.

\bibitem[At-Da]{At-Da}H. Attouch, A. Damlamian. \textit{On multivalued
evolution equations in Hilbert spaces}, Israel J. Math. 12 (1972), 373-390.

\bibitem[Ba]{Ba}V. Barbu. \textit{Nonlinear semigroups and differential
equations in Banach spaces, }Noordhoff, Leyden (1976).

\bibitem[Br1]{Br1}D. Br\'{e}zis. \textit{Classes d'interpolation associ\'{e}s
\`{a} un op\'{e}ratour monotone,} C. R. Math. Acad. Sci. Paris S\'{e}r. A-B
276 (1973), A1553-A1556.

\bibitem[Br2]{Br2}H. Br\'{e}zis.\textit{\ Interpolation classes for monotone
operators}. In \textit{Partial differential equations and related topics},
Lecture Notes in Math. 446 Springer, Berlin (1975).

\bibitem[Br3]{Br3}H. Br\'{e}zis.\textit{ Operateurs maximaux monotones et
semi-groupes de contractions dans les espaces de Hilbert,} North-Holland Math.
Studies 5, Amsterdam (1973).

\bibitem[Du1]{Du1}S. R. Dunbar. \textit{Traveling wave solution of diffusive
Lotka-Volterra equations: a heteroclinic connection in }$%
%TCIMACRO{\U{211d} }%
%BeginExpansion
\mathbb{R}
%EndExpansion
^{4}$, Trans. Amer. Math. Soc. 286 (1984), 557-594.

\bibitem[Du2]{Du2}S. R. Dunbar. \textit{Traveling waves in diffusive
predator-prey equations: periodic orbits and point-to-periodic heteroclinic
orbits}, SIAM J. Appl. Math. 46 (1986), 1057-1078.

\bibitem[Ev]{Ev}L. C. Evans. \textit{Partial Differential Equations}, American
Mathematical Society, U.S.A. (1998).

\bibitem[Hi]{Hi}N. Hirano. \textit{Existence of periodic solutions for
nonlinear evolution equations in Hilbert spaces}, Proc. Amer. Math. Soc. 120
(1994), 185-192.

\bibitem[Il]{Il}T. Ilmanen. \textit{Elliptic regularization and partial
regularity for motion by mean curvature}, Mem. Amer. Math. Soc. 108 (1994), 520:x+90.

\bibitem[Li]{Li}J.-L. Lions. \textit{Sur certaines \'{e}quations paraboliques
non lin\'{e}aires, }Bull. Soc. Math. France 93 (1965), 155-175.

\bibitem[Li-Ma]{Li-Ma}J.-L. Lions, E. Magenes. \textit{Problem\`{e}s aux
limites non homog\`{e}nes et applications}, Travaux et Recherches
Math\'{e}matiques 1, Dunod/Paris (1968).

\bibitem[Mi-Or]{Mi-Or}A. Mielke, M. Ortiz. \textit{A class of minimum
principles for characterizing the trajectories of dissipative systems, }ESAIM
Control Optim. Calc. Var. 14 (2008), 494-516.

\bibitem[Mi-St]{Mi-St}A. Mielke, U. Stefanelli. \textit{Weighted
energy-dissipation functionals for gradient flows}, ESAIM Control Optim. Calc.
Var. 17 (2011), 52-85.

\bibitem[Mu]{Mu}J. D. Murray. \textit{Mathematical biology. II. Spatial models
and biomedical applications. }Tird edition. Interdisciplinary Applied
Mathematics 18 Springer-Verlag, New York (2003).

\bibitem[Mu2]{Mu2}J. D. Murray. \textit{A pre-pattern formation mechanism for
animal coat marking}, J. Theoret. Biol. 88 (1981), 161-199.

\bibitem[Ol]{Ol}O. A. Oleinik. \textit{On a problem of G. Fichera, }Dolk.
Akad. Nauk SSSR 157 (1964), 1297-1300.

\bibitem[Ot]{Ot}M. \^{O}tani. \textit{Nonmonotone perturbations for nonlinear
parabolic equations associated with subdifferential operators, Cauchy
problems,} J. Differential Equations 46 (1982), 268-299.

\bibitem[Ot2]{Ot2}M. \^{O}tani. \textit{On the existence of strong solutions
for }$\mathrm{d}u(t)/\mathrm{d}t+\partial\psi(u(t))-\partial\psi^{2}(u(t))\in
f(t)$, J. Fac. Sci. Univ. Tokyo Sect. IA Math. 24 (1977), 575-605.

\bibitem[Si]{Si}J. Simon. \textit{Compact sets in the space $L^{p}(0,T;B)$,
}Ann. Mat. Pura. Appl. (4) 146 (1987), 65-96.

\bibitem[SWS]{SWS}S. K. Scott, J. Wang, K. Showalter. \textit{Modelling
studies of spiral waves and target patterns in premixed flames, }J. Chem.
Soc., Faraday Trans. 93 (1997), 1733-1739.
\end{thebibliography}
\end{document}